\documentclass[11pt]{article}

\usepackage[mathscr]{eucal}
\usepackage{epsfig,epsf,psfrag}
\usepackage{amssymb,amsfonts,latexsym}
\usepackage{amsmath}
\usepackage{graphicx}
\usepackage{epstopdf}
\usepackage{bm,xcolor,url}
\usepackage{fixltx2e}
\usepackage{array}
\usepackage{verbatim}
\usepackage{algpseudocode, cite}
\usepackage{algorithm}
\usepackage{verbatim}
\usepackage{textcomp}
\usepackage{mathrsfs}
\usepackage[referable]{threeparttablex}
\usepackage{subfig}
\usepackage{amsthm}
\usepackage{enumerate}
\usepackage{footnote}
\usepackage{enumitem}
\usepackage{xr}
\usepackage{mathtools}
\catcode`~=11 \def\UrlSpecials{\do\~{\kern -.15em\lower .7ex\hbox{~}\kern .04em}} \catcode`~=13 

\allowdisplaybreaks[3]

\newcommand{\tnorm}[1]{{\left\vert\kern-0.25ex\left\vert\kern-0.25ex\left\vert #1 
    \right\vert\kern-0.25ex\right\vert\kern-0.25ex\right\vert}}
\newcommand{\tnormt}[1]{{\vert\kern-0.25ex\vert\kern-0.25ex\vert #1 
    \vert\kern-0.25ex\vert\kern-0.25ex\vert}}

\newcommand{\normt}[1]{\Vert#1\Vert}
\newcommand{\abst}[1]{\vert#1\vert}

\newcommand{\nn}{\nonumber}

\newcommand{\nt}{\addtocounter{equation}{1}\tag{\theequation}} 

\newcommand{\dom}{\mathsf{dom}\,}

\newcommand{\ri}{\mathsf{ri}\,}

\newcommand{\calA}{\mathcal{A}}
\newcommand{\calB}{\mathcal{B}}

\newcommand{\calF}{\mathcal{F}}

\newcommand{\calN}{\mathcal{N}}

\newcommand{\calP}{\mathcal{P}}

\newcommand{\calV}{\mathcal{V}}

\newcommand{\calX}{\mathcal{X}}

\newcommand{\bbE}{\mathbb{E}}

\newcommand{\bbR}{\mathbb{R}}

\newcommand{\bbU}{\mathbb{U}}

\newcommand{\bbX}{\mathbb{X}}

\DeclareMathAlphabet{\mathbsf}{OT1}{cmss}{bx}{n}

\newcommand{\rvA}{\mathsf{A}}

\newcommand{\rvP}{\mathsf{P}}

\newcommand{\rvx}{\mathsf{x}}

\newcommand{\rvy}{\mathsf{y}}

\newcommand{\tilG}{\widetilde{G}}

\newcommand{\barD}{\bar{D}}

\newcommand{\barG}{\bar{G}}

\newcommand{\lrangle}[2]{\left\langle{#1},{#2}\right\rangle}
\newcommand{\lranglet}[2]{\langle{#1},{#2}\rangle}

\DeclareMathOperator*{\argmax}{arg\,max}
\DeclareMathOperator*{\argmin}{arg\,min}

\newtheorem{theorem}{Theorem} 
\newtheorem*{theorem*}{Theorem}
\newtheorem{lemma}{Lemma}

\newtheorem{corollary}{Corollary}
 
\newtheorem*{assump*}{Assumption}

\theoremstyle{remark}
\newtheorem{remark}{Remark}

\newcommand{\qednew}{\nobreak \ifvmode \relax \else
      \ifdim\lastskip<1.5em \hskip-\lastskip
      \hskip1.5em plus0em minus0.5em \fi \nobreak
      \vrule height0.75em width0.5em depth0.25em\fi}

\numberwithin{equation}{section}
\numberwithin{assump}{section}
\numberwithin{lemma}{section}
\numberwithin{theorem}{section}
\numberwithin{prop}{section}
\numberwithin{corollary}{section}
\numberwithin{remark}{section}
\usepackage{graphicx}
\usepackage[colorlinks=true,breaklinks=true,bookmarks=true,urlcolor=blue,
     citecolor=blue,linkcolor=blue,bookmarksopen=false,draft=false]{hyperref}

\newcommand{\rx}{r_\rvx}

\newcommand{\ry}{r_\rvy}

\newcommand{\hx}{h_\rvx}
\newcommand{\hy}{h_\rvy}

\newcommand{\Lx}{L_\rvx}
\newcommand{\Ly}{L_\rvy}
\newcommand{\zetax}{\zeta_\rvx}
\newcommand{\zetay}{\zeta_\rvy}

\usepackage{fullpage}
\usepackage{caption}

\begin{document}

\title{A Generalized Frank-Wolfe Method With ``Dual Averaging'' for Strongly Convex Composite Optimization}

\author{Renbo Zhao\thanks{MIT Operations Research Center, 77 Massachusetts Avenue, Cambridge, MA   02139 (\href{mailto: renboz@mit.edu}{renboz@mit.edu}).} \quad Qiuyun Zhu\thanks{Department of Mathematics and Statistics, Boston University, 111 Cummington Mall, Boston, MA 02215 (\href{mailto:  rachyun@bu.edu}{rachyun@bu.edu}).}}


\maketitle

\begin{abstract}
We propose a simple variant of the generalized Frank-Wolfe method for solving strongly convex composite optimization problems, by introducing an additional averaging step on the dual variables. We show that in this variant,  one can choose a simple constant step-size and obtain a linear convergence rate on the duality gaps. By leveraging the convergence analysis of this variant, we then analyze the local convergence rate of the logistic fictitious play algorithm, which is well-established in game theory but lacks any form of convergence rate guarantees. We show that, with high probability, this algorithm converges locally at rate $O(1/t)$, in terms of certain expected duality gap.  
\end{abstract}


\section{Introduction}\label{sec:intro}

Given two finite-dimensional real normed spaces $(\bbX,\normt{\cdot}_{\bbX})$ and $(\bbU,\normt{\cdot}_{\bbU})$, with dual spaces denoted by $(\bbX^*,\normt{\cdot}_{\bbX^*})$ and $(\bbU^*,\normt{\cdot}_{\bbU^*})$, respectively,  let us consider the following convex optimization problem:
\begin{equation}
p^*:= {\min}_{x\in\bbX}\;\;\; [p(x):=f(\rvA x) + h(x)], \tag{P} \label{eq:P}
\end{equation}
where $\rvA:\bbX\to\bbU$ is  a (bounded) linear operator,  $f:\bbU\to\bbR$ is a convex differentiable function whose gradient is $L$-Lipschitz on $\bbU$ (for some $L>0$), namely, 
\begin{equation}
\normt{\nabla f(u) - \nabla f(v)}_{\bbU^*}\le L\normt{u-v}_{\bbU}, \quad \forall\,u,v\in\bbU, 
\end{equation}
and $h:\bbX\to\bbR\cup\{+\infty\}$ is a closed convex function with nonempty domain, which is denoted by $\dom h:=\{x\in\bbX:h(x)<+\infty\}$. 
We assume that $h$ is ``simple'' such that  the ``generalized'' linear optimization sub-problem, namely 
\begin{equation}
{\min}_{x\in\bbX}\;\;\; \lrangle{c}{x} + h(x), \tag{GLO} \label{eq:LMO}
\end{equation}
can be easily solved for  any $c\in\bbX^*$. As an example, if $h$ is an indicator function of a polytope $\calP$, then~\eqref{eq:LMO} becomes a linear program, which has an optimal solution that is a vertex of $\calP$. 

\begin{algorithm}[t!]
\caption{GFW Method 
}\label{algo:GFW}
\begin{algorithmic}
\State {\bf Input}: Starting point $x^0\in\dom h$, and step-sizes $\{\alpha_t\}_{t\ge 0}\subseteq (0,1]$ 
\State {\bf At iteration $t\in\{0,1,\ldots\}$}:
\begin{enumerate}
{\setlength\itemindent{10pt} \item \label{item:LMO}  Compute $\nabla f(\rvA x^t)$ and $v^t\in\calV(x^t):=\argmin_{x\in\bbX}\; \lranglet{\rvA^* \nabla f(\rvA x^t)}{x} + h(x)$} 
{\setlength\itemindent{10pt} \item Update $x^{t+1}:= x^t + \alpha_t (v^t-x^t)$ 
}
\end{enumerate}
\end{algorithmic}
\end{algorithm}

First-order methods that involve computing the gradients of $f$ and solving sub-problems in
the form of~\eqref{eq:LMO} are referred to as the {\em generalized} Frank-Wolfe (GFW) method, which is shown Algorithm~\ref{algo:GFW}.  This method have been studied in several previous works,  
e.g., Bach~\cite{Bach_15}, Nesterov~\cite{Nest_18}, Ghadimi~\cite{Ghad_19} and Pena~\cite{Pena_21} (all of which we will review shortly in Section~\ref{sec:prior_art}).  
Indeed, the name of the GFW method comes from the fact that it can be regarded as a generalization of the Frank-Wolfe (FW) method~\cite{Frank_56}, which dates back to 1950s (see~\cite{Freund_16} and references therein). Specifically, 
if $h$ is the indicator function of some ``simple'' convex compact set (under which~\eqref{eq:LMO} can be easily solved), then the GFW method specializes precisely to the FW method. 

\subsection{Review of Computational Guarantees of the GFW Method} \label{sec:prior_art}

When $h$ is non-strongly-convex on $\dom h$, the computational guarantees of the GFW method (i.e., Algorithm~\ref{algo:GFW}) are exactly the same as those of the classical FW methods (see e.g.,~\cite{Freund_16}). 
Specifically, assume $\dom h$ to be bounded, and define the diameter of $\rvA (\dom h):= \{\rvA x: x\in\dom h\}$ as 
$$\barD:= {\sup}_{x,y\in\dom h} \;\normt{\rvA(x-y)}_\bbU<+\infty.$$ Additionally, for all $t\ge 0$, let us define the primal optimality gap at $x^t$ by $$\delta(x^t):= p(x^t) - p^*$$ and the FW-gap at $x^t$ by 
\begin{equation}
G(x^t):= \lranglet{\rvA^*\nabla f(\rvA x^t)}{x^t - v^t} + h(x^t) - h(v^t). \label{eq:G} 
\end{equation}
In fact,  it is easy to see that: 
\begin{enumerate}[label=\roman*)]
\item  $G(x^t)\ge \delta(x^t)\ge 0$ for all $t\ge 0$ (cf.~\cite[Eqn.~(2.4)]{Zhao_20b}), 
\item  $G(x^t)$ has the same value for any $v^t\in\calV(x^t)$ (cf.\ Step~\ref{item:LMO} in Algorithm~\ref{algo:GFW}), and hence the particular choice of $v^t\in\calV(x^t)$ in~\eqref{eq:G} does not matter. 
\end{enumerate}
By choosing the adaptive step-sizes $\alpha_t=\min\{1,G(x^t)/(L\barD^2)\}$ or the predetermined step-sizes $\alpha_t = 2/(t+2)$ for $t\ge 0$, one can show that (see e.g.,~\cite[Section 4]{Bach_15})
\begin{align}
\delta(x^t) &\le \frac{2L\barD^2}{t+1} \quad\mbox{and}\quad  \barG_t:=\min_{0 \le i \le t} G(x^i)
 \le \frac{8L\barD^2}{t+1},\quad \forall\,t\ge 0.\label{eq:FW_h_nsc}
\end{align}

Next, let us consider the case where $h$ is $\mu$-strongly-convex on $\dom h$ (for some $\mu>0$), namely
\begin{equation}
h(\lambda x + (1-\lambda) y) \le \lambda h(x) + (1-\lambda) h(y) - (\mu/2)\lambda(1-\lambda)\normt{x-y}_\bbX^2, \quad \forall\,x,y\in\dom h, \;\forall\lambda\in[0,1].\nn
\end{equation}
In contrast to the well-studied non-strongly-convex case, the computational guarantees of the GFW method (i.e., Algorithm~\ref{algo:GFW}) in the strongly-convex case appear to be less studied. 
Nesterov~\cite[Section 5]{Nest_18} showed that, if $\dom h$ is bounded with diameter  $$D:= {\sup}_{x,y\in\dom h} \;\normt{x-y}_\bbX<+\infty,$$ and one chooses the predetermined step-sizes $\alpha_t = \frac{6(t+1)}{(t+2)(2t+3)}$ for all $t\ge 0$, then 
\begin{equation}
\delta(x^t)\le \frac{27\kappa^2\mu D^2}{(t+1)(2t+1)}, \quad \forall\,t\ge 0, 
\qquad \mbox{where}\quad \kappa:= \frac{L\normt{\rvA}^2}{\mu}  \label{eq:FW_Nest}
\end{equation}
denotes the condition number of~\eqref{eq:P}, and $\normt{\rvA}: = {\sup}_{\normt{x}_\bbX=1}\;\normt{\rvA x}_\bbU$
denotes the operator norm of $\rvA$. 
Later on, Ghadimi~\cite[Corollary~1(b)]{Ghad_19} showed that without assuming the boundedness of $\dom h$, as long as one chooses the constant step-sizes $\alpha_t = 1/(1+4\kappa)$ for $t\ge 0$, then the sequence of minimum FW-gaps $\{\barG_t\}_{t\ge 0}$ (cf.~\eqref{eq:FW_h_nsc}) converges to zero linearly: 
\begin{equation}
\barG_t\le 4\delta(x^0) (1+4\kappa)\big(1-1/(2(1+4\kappa))\big)^t, \quad \forall\, t\ge 1. \label{eq:FW_Ghad}
\end{equation}
In addition, Ghadimi~\cite[Corollary~2]{Ghad_19} showed that the convergence rate in~\eqref{eq:FW_Ghad} continues to hold (up to absolute constants) if one instead chooses the step-sizes $\{\alpha_t\}_{t\ge 0}$ via certain backtracking line-search procedure. More recently, 
Pena~\cite{Pena_21} 
showed that by choosing the step-sizes $\{\alpha_t\}_{t\ge 0}$ via exact line-search, we have the following simpler linear convergence result: 
\begin{equation}
\tilG_t \le \tilG_0\big(1-(1/2)\min\{1,(2\kappa)^{-1}\}\big)^t, \quad \forall\,t\ge 0,  \label{eq:FW_Pena}
\end{equation}  
where 
\begin{equation}
\tilG_t := f(\rvA x^t) + h(x^t) + {\min}_{0\le i \le t}\; f^*(\nabla f(x^i)) + h^*(-\rvA^*\nabla f(x^i)), \quad \forall\,t\ge 0. \label{eq:tilG} 
\end{equation}
In~\eqref{eq:tilG}, $f^*$ and $h^*$ denote Fenchel conjugates of $f$ and $h$, respectively, and $\rvA^*:\bbU^*\to\bbX^*$ denotes the adjoint of $\rvA$  (see Section~\ref{sec:prelim} for details). 
In addition,  Pena~\cite[Theorem~2]{Pena_21} showed that~\eqref{eq:FW_Pena} still holds (up to absolute constants) if  the step-sizes $\{\alpha_t\}_{t\ge 0}$ are chosen by  backtracking line-search. 

\subsection{Main Contributions}

\begin{algorithm}[t!]
\caption{GFW Method with Dual Averaging for Strongly Convex $h$ 
}\label{algo:GFWDA}
\begin{algorithmic}
\State {\bf Input}: Starting point $x^0\in\dom h$, $y^0\in \dom f^*$, and step-sizes $\{\alpha_t\}_{t\ge 0}\subseteq (0,1]$. 
\State {\bf At iteration $t\in\{0,1,\ldots\}$}:
\begin{enumerate}
{\setlength\itemindent{10pt} \item \label{item:LMO2}  Compute 
$v^t=\argmin_{x\in\bbX}\; \lranglet{\rvA^* y^t}{x} + h(x)$} 
{\setlength\itemindent{10pt} \item Update $x^{t+1}:= x^t + \alpha_t (v^t-x^t)$\label{item:x_upd}}
{\setlength\itemindent{10pt} \item Compute $g^t:=\nabla f(\rvA x^t)$ and update $y^{t+1}:= (1-\alpha_t)y^t + \alpha_t g^t$\label{item:DA}} 
\end{enumerate}
\end{algorithmic}
\end{algorithm}

In this work, we focus on the case where $h$ is $\mu$-strongly convex (for some $\mu>0$), and propose a simple 
variant of the GFW method (cf.~Algorithm~\ref{algo:GFW}) in Algorithm~\ref{algo:GFWDA}.  
Compared with Algorithm~\ref{algo:GFW}, we see that it simply adds an additional averaging step to the 
dual iterates $\{y^t\}_{t\ge 0}$ (cf.\ Step~\ref{item:DA}), and the weights of averaging are exactly given by the primal step-sizes $\{\alpha_t\}_{t \ge 0}$. 
As a result, the primal iterates $\{x^t\}_{t\ge 0}$ and the dual iterates $\{y^t\}_{t\ge 0}$ are updated in a ``symmetric'' fashion. Despite the simplicity of this additional ``dual averaging'' step, as we will show in Section~\ref{sec:analysis}, it allows us to establish {\em a simple and elegant contraction} on the sequence of duality gaps evaluated on the primal-dual pairs $\{(x^t,y^t)\}_{t\ge 0}$, by simply choosing the step-sizes $\{\alpha_t\}_{t\ge 0}$ to be a constant (that only depends on the condition number $\kappa$). Compared with the previous results established for the GFW method in Algorithm~\ref{algo:GFW}~\cite{Nest_18,Ghad_19,Pena_21}, 
our result shows that  Algorithm~\ref{algo:GFWDA} has the benefit of 
a simple choice of step-sizes, which does not involve any form of line-search --- this feature is particularly attractive when the condition number $\kappa$ (cf.~\eqref{eq:FW_Nest}) is explicitly known or can be easily estimated.   

As the second main contribution of this work, we analyze the local convergence rate of the logistic fictitious play (LFP) algorithm~\cite{Fuden_93}, which is a classical stochastic game-theoretic algorithm that has only been shown to converge asymptotically. 
The key to our analysis is to observe that the deterministic version of LFP (D-LFP) is a special instance of Algorithm~\ref{algo:GFWDA}, and LFP can be regarded as a certain ``stochastic approximation'' of D-LFP. As a result, by properly incorporating the ``stochastic noise'' in LFP to the analysis of D-LFP (which is precisely that of Algorithm~\ref{algo:GFWDA}), we then obtain the local convergence rate of LFP. Our analysis shows that, with high probability, LFP converges locally at rate $O(1/t)$ in terms of certain expected duality gap. 
Somewhat surprisingly, our numerical results 
on the empirical behavior of LFP are in excellent consistence with our theory.\\  


\noindent {\bf Notations.} For any non-empty set $\calX$, we denote its relative interior as $\ri \calX$. In addition, let $\iota_\calX$ denote its indicator function (namely, $\iota_\calX(x) = 0$ if $x\in\calX$ and $\iota_\calX(x) = +\infty$ otherwise).  For a matrix $M$ and any $p,q\in[1,+\infty]$, we define its $(p,q)$-operator-norm as $\normt{M}_{p,q}=\sup_{\normt{z}_p=1}\,\normt{Mz}_q$. 

\section{Preliminaries}\label{sec:prelim}

Let us provide some background on duality theory that will be useful in our analysis in Section~\ref{sec:analysis}. In the rest of this work, for notational brevity, we will omit the subscript of norms, and the meaning of $\normt{\cdot}$ and $\normt{\cdot}_*$  can be inferred from the context. 

To begin with, let us first write down the (Fenchel) dual problem associated with~\eqref{eq:P}:
\begin{equation}
-d^*:=-{\min}_{y\in\dom f^*} \;\; [d(y):=f^*(y) + h^*(-\rvA^*y)] = {\max}_{y\in\dom f^*} - d(y), \tag{D} \label{eq:D}
\end{equation}
where recall that $\rvA^*:\bbU^*\to\bbX^*$ denotes the adjoint of $\rvA$, and $f^*:\bbU^*\to\bbR\cup\{+\infty\}$ and $h^*:\bbX^*\to\bbR$ denote the Fenchel conjugates of $f$ and $h$, respectively: 
\begin{align}
f^*(y)&:= {\sup}_{u\in\bbU}\; \lranglet{y}{u} - f(u), \;\;\;\;\;\quad\forall\,y\in\bbU^*,\label{eq:def_f*}\\ 
h^*(z)&: = {\sup}_{x\in\dom h}\; \lranglet{z}{x} - h(x), \;\;\;\;\forall\,z\in\bbX^*.  \label{eq:def_h*}
\end{align}
 From standard results (see e.g.,~\cite{Kakade_09a}), we know that 
\begin{enumerate}[label=\roman*),leftmargin=3ex]
\item The function $f^*$ is $(1/L)$-strongly convex on $\dom f^*$ in the following sense:
\begin{equation}
f^*(y)\ge f^*(w) + \lranglet{g}{y-w} + (2L)^{-1} \normt{y-w}^2, \;\; \forall\,y\in\dom f^*,\;\forall\,w\in\dom \partial f^*, \;\forall\,g\in \partial f^*(w), \label{eq:sc_f*}
\end{equation}
where $\dom \partial f^*:= \{y\in\dom f^*: \partial f^*(y)\ne \emptyset\}$ denotes the set of sub-differentiable points of $f^*$.
\item 
The function $h^*$ is convex and differentiable on $\bbX^*$ and $\nabla h^*$ is $(1/\mu)$-Lipschitz on $\bbX^*$. 
\end{enumerate}
In addition, from~\cite[Theorem 3.51]{Peyp_15}, we see that strong duality holds between~\eqref{eq:P} and~\eqref{eq:D}, namely $p^* = -d^*$. Next, let us define the duality gap $\Delta:\dom h\times \dom f^*\to \bbR$ as 
\begin{equation}
\Delta(x,y):= p(x) + d(y) = f(\rvA x) + h(x) + f^*(y) + h^*(-\rvA^*y), \quad \forall\,x\in \dom h, \;\; \forall\,y\in \dom f^*. \label{eq:pd_gap}
\end{equation}
Using standard results in Fenchel duality (see e.g.,~\cite{Bach_15}), we see that for any $x\in\dom h$, 
\begin{equation}
\Delta(x,\nabla f(\rvA x)) = G(x)= \lranglet{\rvA^*\nabla f(\rvA x)}{x - v(x)} + h(x) - h(v(x)),\label{eq:FWgap_DualityGap}
\end{equation}
where $v(x):= \argmin_{x'\in\bbX}\; \lranglet{\nabla f(\rvA x)}{\rvA x'} + h(x')$ and $G:\dom h\to\bbR$ is defined in~\eqref{eq:G}. In words, for all $x\in\dom h$,~\eqref{eq:FWgap_DualityGap} means that the duality gap at $(x,\nabla f(\rvA x))$ is equal to the FW-gap at $x$.

\section{Convergence Rate of Algorithm~\ref{algo:GFWDA}}\label{sec:analysis}

We derive the (global) convergence rate of Algorithm~\ref{algo:GFWDA} in the following theorem. 


\begin{theorem}
\label{thm:lin_conv}
In Algorithm~\ref{algo:GFWDA}, if we choose $\alpha_t = \min\{1/(2\kappa),1\}$ for all $t\ge 0$,  then 
\begin{equation}
\Delta(x^t,y^t)\le \rho(\kappa)^t \Delta(x^0,y^0), \quad \forall\,t\ge 0, \label{eq:lin_rate}
\end{equation}
where $\kappa$ denotes the condition number of~\eqref{eq:P} $($cf.~\eqref{eq:FW_Nest}$)$ and 
the linear rate function $\rho:[0,+\infty)\to [0,1)$ is defined as 
\begin{equation}
\rho(\kappa):= \begin{cases}
\kappa, & \mbox{if} \;\; 0\le \kappa \le  1/2\\
1-{1}/(4\kappa), & \mbox{if} \;\; \kappa> 1/2
\end{cases}. 
\end{equation}
\end{theorem}

\begin{proof}
First, from the definition of $v^t$ in Step~\ref{item:LMO2}, we see that $$v^t = {\argmax}_{x\in\bbX}\; \lranglet{-\rvA ^* y^t}{x} - h(x), $$ and hence from the definition of $h^*$ in~\eqref{eq:def_h*}, we have 
\begin{equation}
h^*(-\rvA ^* y^t) = \lranglet{-\rvA ^* y^t}{v^t} - h(v^t). 
\label{eq:ineq1}
\end{equation}
Since both $f$ and $h^*$ have Lipschitz gradients on $\bbU$ and $\bbX^*$, respectively,  we have 
\begin{align}
f(\rvA x^{t+1})&\le f(\rvA x^{t}) + \alpha_t\lranglet{\nabla f(\rvA x^{t})}{\rvA(v^t-x^t)} + (L\alpha_t^2/2) \normt{\rvA(v^t-x^t)}^2\nn\\
 &\le f(\rvA x^{t}) + \alpha_t\lranglet{g^t}{\rvA(v^t-x^t)} + (L\alpha_t^2\normt{\rvA}^2/2) \normt{v^t-x^t}^2,\label{eq:f} \\
 h^*(-\rvA^* y^{t+1})&\le h^*(-\rvA^* y^{t}) + \alpha_t\lranglet{\nabla h^*(-\rvA^* y^{t})}{\rvA^*(y^t-g^t)} + (\alpha_t^2/(2\mu)) \normt{\rvA^*(y^t-g^t)}^2\nn\\
 &\le h^*(-\rvA^* y^{t}) + \alpha_t\lranglet{v^t}{\rvA^*(y^t-g^t)} + (\alpha_t^2\normt{\rvA}^2/(2\mu)) \normt{y^t-g^t}^2. \label{eq:h*}
\end{align}
In addition, by the convexities of $h$ and $f^*$ on their respective domains, we have
\begin{align}
h(x^{t+1})&\le (1-\alpha_t) h(x^t) + \alpha_t h(v^t), \label{eq:h}\\
f^*(y^{t+1})&\le (1-\alpha_t) f^*(y^t) + \alpha_t f^*(g^t). \label{eq:f*}
\end{align}
Combining~\eqref{eq:f} to~\eqref{eq:f*}, we have
\begin{align}
\Delta(x^{t+1},y^{t+1}) &= f(\rvA x^{t+1}) + h(x^{t+1}) + h^*(-\rvA^* y^{t+1}) + f^*(y^{t+1})\nn\\
&\le  \Delta(x^{t},y^{t}) -\alpha_t\big\{\big[\lranglet{\rvA x^t}{g^t} - f^*(g^t)\big]+ \big[\lranglet{-\rvA^* y^t}{v^t} - h(v^t)\big] + f^*(y^t)+ h(x^t) \big\}\nn\\
&\hspace{2cm} + \alpha_t^2\kappa\big\{(\mu/2)\normt{v^t-x^t}^2+(2L)^{-1}\normt{y^t-g^t}^2\big\}. \label{eq:bd_Delta_t+1}
\end{align}
Since 
\begin{equation}
g^t = \nabla f(\rvA x^t) = {\argmax}_{y\in\bbU^*}\; \lranglet{\rvA x^t}{y} - f^*(y),\label{eq:def_g^t}
\end{equation} 
we see that $g^t\in \dom \partial f^*$ and 
\begin{align}
f(\rvA x^t) = \lranglet{\rvA x^t}{g^t} - f^*(g^t). \label{eq:ineq3}
\end{align}
By substituting~\eqref{eq:ineq1} and~\eqref{eq:ineq3} into~\eqref{eq:bd_Delta_t+1}, we see that
\begin{align}
\Delta(x^{t+1},y^{t+1}) \le (1-\alpha_t) \Delta(x^{t},y^{t}) 
+ \alpha_t^2\kappa\big\{(\mu/2)\normt{v^t-x^t}^2+(2L)^{-1}\normt{y^t-g^t}^2\big\}. \label{eq:bd_Delta_t+1}
\end{align}
By the $\mu$-strong convexity of $h$ on its domain, the definition of $v^t$ in Step~\ref{item:LMO2} and~\eqref{eq:ineq1}, we have 
\begin{align}
 (\mu/2)\normt{x^t-v^t}^2&\le h(x^t)-h(v^t)+\lranglet{y^t}{\rvA (x^t-v^t)}\nn\\
 & = h(x^t) + \lranglet{y^t}{\rvA x^t} + h^*(-\rvA ^* y^t).\label{eq:ineq2}
\end{align}
In addition, using the $(1/L)$-strong convexity of $f^*$ in the sense of~\eqref{eq:sc_f*},~\eqref{eq:def_g^t} and~\eqref{eq:ineq3}, we have 
\begin{align}
(2L)^{-1}\normt{y^t-g^t}^2 &\le f^*(y^t)-f^*(g^t)+\lranglet{\rvA x^t}{g^t-y^t}\nn\\ 
&= f^*(y^t)-\lranglet{\rvA x^t}{y^t} + f(\rvA x^t). \label{eq:ineq4}
\end{align}
Substituting~\eqref{eq:ineq2} and~\eqref{eq:ineq4} into~\eqref{eq:bd_Delta_t+1}, we have
\begin{equation}
\Delta(x^{t+1},y^{t+1})\le (1-\alpha_t + \alpha_t^2\kappa)\Delta(x^{t},y^{t}). 
\end{equation}
If we choose $\alpha_t = \min\{1/(2\kappa),1\}$, then we see that $1-\alpha_t + \alpha_t^2\kappa = \rho(\kappa)$, for all $t\ge 0$. 
\end{proof}

\begin{remark} \label{rmk:cst_det_step}
Note that in Theorem~\ref{thm:lin_conv}, the linear rate function $\rho$ is continuous, concave and strictly increasing on $[0,+\infty)$. 
Hence the smaller the condition number $\kappa$, the better the linear rate. In the regime that $\kappa>1/2$, we have $\rho(\kappa) = 1-1/(4\kappa)$, and therefore, 
to find a primal-dual pair $(x,y)\in\dom h\times \dom f^*$ such that $\Delta(x,y)\le \varepsilon$,  it requires no more than 
\begin{equation}
\left\lceil 4\kappa\ln\left(\frac{\Delta(x^0,y^0)}{\varepsilon}\right)\right\rceil 
\quad \mbox{iterations}.
\end{equation}
\end{remark}

%
%

\section{Application to LFP}


The LFP algorithm (a.k.a.\ stochastic fictitious play with best logit response), first introduced by Fudenberg and Kreps in 1993~\cite{Fuden_93}, is a classical algorithm in game theory (see~\cite{Ny_06} and references therein). In this work we focus on the two-player zero-sum version, which is shown  in Algorithm~\ref{algo:LFP}. 
Specifically, 
players I and II are given finite action spaces $[n]:=\{1,2,\ldots,n\}$ and $[m]$, respectively, and a payoff matrix $A\in\bbR^{m\times n}$. At the beginning, 
players I and II choose their initial actions $i_0\in[n]$ and $j_0\in[m]$, respectively, and their initial ``history of actions'' are denoted by $x^0 := e_{i_0}$ and $y^0 := e_{j_0}$, respectively (where $e_i$ denotes the $i$-th standard coordinate vector). 
At any time $t\ge 0$, in order for player I 
to choose the next action $i_{t+1}$, 
she first computes the best-logit-response distribution $w^t$ based on player II's history of actions $y^t$:
\begin{align}
w^t:= \rvP_\rvx(y^t):= {\argmin}_{x\in\Delta_n} \; \lranglet{A^\top y^t}{x} + \eta h_\rvx(x),\label{eq:def_wt}
\end{align}
where 
\begin{equation}
\hx(x):= \textstyle\sum_{i=1}^n x_i\ln(x_i) \qquad (\mbox{for }\,x\ge 0)
\end{equation}
is the (negative) entropic function defined on $\bbR^n_+$, 
$\Delta_n:= \{x\in\bbR_+^n: \sum_{i=1}^n  x_i=1\}$ denotes the $(n-1)$-dimensional probability simplex, and $\eta>0$ is the regularization parameter.  
Then, based on $w^t$, she randomly chooses $i_{t+1}$ by sampling from the distribution $w^t$, such that $\Pr(i_{t+1} = i) = w^t_{i}$ for $i\in[n]$. After obtaining $i_{t+1}$, she updates her history of actions from $x^t$ to $x^{t+1}$ by a convex combination of $x^t$ and $e_{i_{t+1}}$: 
\begin{equation}
x^{t+1} := (1-\alpha_t)x^t + \alpha_t e_{i_{t+1}},
\end{equation}
where $\alpha_t\in[0,1]$ can be interpreted as the ``step-size'' of player I, and is required to satisfy
\begin{equation}
\textstyle\sum_{t=0}^{+\infty} \,\alpha_t = +\infty\quad \mbox{and}\quad \textstyle\sum_{t=0}^{+\infty} \,\alpha_t^2 < +\infty.  \label{eq:alpha_t_cond}
\end{equation}
For player II, the update of her history of actions from $y^t$ to $y^{t+1}$ is symmetric to that of player I. Specifically, based on player I's history of actions $x^t$, she computes her best-logit-response distribution $s^t$ as 
\begin{equation}
s^t:= \rvP_\rvy(x^t):={\argmax}_{y\in\Delta_m} \; \lranglet{A x^t}{y} - \eta h_\rvy(y),  \label{eq:def_st}
\end{equation}
where 
\begin{equation}
\hy(y):= \textstyle\sum_{j=1}^m y_j\ln(y_j) \qquad (\mbox{for }\,y\ge 0)
\end{equation}
is the (negative) entropic function defined on $\bbR^m_+$. 
Then she samples her next action $j_{t+1}$ from $s^t$, and updates her history of actions from $y^t$ to $y^{t+1}$ as follows:
\begin{equation}
y^{t+1} := (1-\alpha_t)y^t + \alpha_t e_{j_{t+1}}.
\end{equation}

\begin{algorithm}[t!]
\caption{Logistic fictitious play (LFP)}\label{algo:LFP}
\begin{algorithmic}
\State {\bf Input}: Initial actions $i_0\in[n]$ and $j_0\in[m]$, and and step-sizes $\{\alpha_t\}_{t\ge 0}\subseteq (0,1]$. 
\State {\bf Initialize}: $x^0 := e_{i_0}$ and $y^0 := e_{j_0}$ 
\State {\bf At iteration $t\in\{0,1,\ldots\}$}:
\begin{align}
\label{eq:upd_x}
\begin{split}
x^{t+1} &:= (1-\alpha_t)x^t + \alpha_t e_{i_{t+1}}, \quad\mbox{where}\;\;  i_{t+1}\sim w^t \; \mbox{and}\;\; w^t \mbox{ is defined in }\eqref{eq:def_wt},
\end{split}\\
\label{eq:upd_y}
\begin{split}
y^{t+1} &:= (1-\alpha_t)y^t + \alpha_t e_{j_{t+1}},\quad\mbox{where}\;\;  j_{t+1}\sim s^t\;\; \mbox{and}\;\; s^t\mbox{ is defined in }\eqref{eq:def_st} . 
\end{split} 
\end{align}
\end{algorithmic}
\end{algorithm}

In the literature, the asymptotic convergence of LFP (i.e., Algorithm~\ref{algo:LFP}) has been well-studied. For example, Hofbauer and Sandholm~\cite{Hof_02} proves the following theorem:

\begin{theorem}[Hofbauer and Sandholm~{\cite[Theorem~6.1(ii)]{Hof_02}}]\label{thm:asymp_as}
Consider the fixed-point equation 
\begin{equation}
\rvP_\rvx(y) = x, \quad \rvP_\rvy(x) = y,\label{eq:fixed_pt}
\end{equation}
and its unique solution $(x^*,y^*)\in \ri\Delta_n\times\ri\Delta_m$. 
Then in Algorithm~\ref{algo:LFP}, for any initialization $i_0\in[n]$ and $j_0\in[m]$, and any step-sizes $\{\alpha_t\}_{t\ge 0}$ satisfying~\eqref{eq:alpha_t_cond},  we have 
\begin{equation}
\Pr\left({\lim}_{t\to+\infty}\, x_t = x^* \;\; \mbox{and}\;\; {\lim}_{t\to+\infty}\, y_t = y^*\right) = 1. \label{eq:as_conv}
\end{equation}
\end{theorem}

\begin{algorithm}[t!]
\caption{Deterministic version of LFP (D-LFP)}\label{algo:det_LFP}
\begin{algorithmic}
\State {\bf Input}: Starting point $x^0 \in\Delta_n$ and $y^0 \in\Delta_m$, and step-sizes $\{\alpha_t\}_{t\ge 0}\subseteq (0,1]$. 
\State {\bf At iteration $t\in\{0,1,\ldots\}$}:
\begin{align}
x^{t+1}& := (1-\alpha_t)x^t + \alpha_t w^t,\quad\mbox{where}\;\;   w^t \mbox{ is defined in }\eqref{eq:def_wt},\label{eq:upd_x_exp}\\
y^{t+1}& := (1-\alpha_t)y^t + \alpha_t s^t, \;\quad\mbox{where}\;\;   s^t\mbox{ is defined in }\eqref{eq:def_st} . \label{eq:upd_y_exp}
\end{align}
\end{algorithmic}
\end{algorithm}

However, in contrast to the well-understanding of the asymptotic convergence of LFP, the  convergence rate of LFP is largely unknown. 
  The purpose of this section is to conduct a local convergence rate analysis of LFP, and show that with high probability, LFP converges locally at rate $O(1/t)$, where the convergence is measured in terms of certain expected duality gap. 
To that end, let us first consider D-LFP (i.e., the deterministic version of LFP), which is shown in Algorithm~\ref{algo:det_LFP}, and relate it to Algorithm~\ref{algo:GFWDA}. 

\subsection{Relating D-LFP to Algorithm~\ref{algo:GFWDA}} \label{sec:D_LFP}

Let us observe a simple but important fact,  that  is, D-LFP (i.e., Algorithm~\ref{algo:det_LFP}) is  an instance of Algorithm~\ref{algo:GFWDA} for solving the following instance of~\eqref{eq:P}:
\begin{equation}
\min_{x\in\bbR^n}\; \underbrace{\eta\ln\big(\textstyle\sum_{j=1}^m \exp(a_j^\top x/\eta)\big)}_{:= f(\rvA x)} + \underbrace{\eta\hx(x)+\iota_{\Delta_{n}}(x)}_{:=h(x)},  \tag{P-LFP}\label{eq:P_LFP}
\end{equation}
where $a_j^\top$ denotes the $j$-th row of $A$ (for $j\in[m]$),  
the linear operator $\rvA x:= Ax $ for $x\in\bbR^n$, and
\begin{equation}
f(u):= \eta\ln(\textstyle\sum_{j=1}^m \exp(u_j/\eta)), \quad \forall\,u\in\bbR^m.  \label{eq:log-sum-exp}
\end{equation}
As a result, the dual problem of~\eqref{eq:P_LFP} reads:
\begin{equation}
-\min_{y\in\bbR^m}\; \underbrace{\eta\hy(y)+\iota_{\Delta_{m}}(y)}_{:=f^*(y)} + \underbrace{\eta\ln\big(\textstyle\sum_{i=1}^n \exp(-A_i^\top y/\eta)\big)}_{:= h^*(-A^\top y)} , \tag{D-LFP} \label{eq:D_LFP}
\end{equation}
where $A_i$ denotes the $i$-th column of $A$ (for $i\in[n]$). 
Consequently, according to~\eqref{eq:pd_gap},
the duality gap $\Delta:\Delta_n\times\Delta_m\to\bbR$ has the following form: for any $(x,y)\in \Delta_n\times\Delta_m,$
\begin{equation}
\Delta(x,y) = \ln\big(\textstyle\sum_{j=1}^m \exp(a_j^\top x/\eta)\big) + \eta\hx(x) + \eta\ln\big(\textstyle\sum_{i=1}^n \exp(-A_i^\top y/\eta)\big) + \eta\hy(y). \label{eq:Delta_LFP}
\end{equation}

Now, in order to see that~\eqref{eq:P_LFP} is an instance of~\eqref{eq:P} (which in turn implies that~\eqref{eq:D_LFP} is an instance of~\eqref{eq:D}), it suffices to note that
\begin{enumerate}[label=\roman*)]
\item The function $f$ given in~\eqref{eq:log-sum-exp} is convex and differentiable on $\bbR^m$, and $\nabla f$ is $(1/\eta)$-Lipschitz on $\bbR^m$ with respect to $\normt{\cdot}_\infty$, i.e.,
\begin{equation}
\normt{\nabla f(u) - \nabla f(u')}_1\le (1/\eta)\normt{u-u'}_{\infty}, \quad \forall\,u,u'\in\bbR^m.  
\end{equation}
(To see this, note that for all $u\in\bbR^m$, $\normt{\nabla^2 f(u)}_{\infty,1}:= {\sup}_{\normt{z}_{\infty}=1}\; \normt{\nabla^2 f(u)z}_1\le 1/\eta.) $
\item The function $h := \eta\hx + \iota_{\Delta_n}$ is $\eta$-strongly convex on  $\dom h = \Delta_n$ with respect to $\normt{\cdot}_1$, i.e.,
\begin{equation}
h(\lambda x + (1-\lambda) y) \le \lambda h(x) + (1-\lambda) h(y) - (\eta/2)\lambda(1-\lambda)\normt{x-y}_1^2, \quad \forall\,x,y\in\Delta_n, \;\forall\lambda\in[0,1].\nn
\end{equation}
(For details, see e.g.,~\cite[Lemma~3]{Nest_05}.)
\end{enumerate}
In addition, to see that D-LFP (i.e., Algorithm~\ref{algo:det_LFP}) is an instance of Algorithm~\ref{algo:GFWDA}, we simply note that for all $t\ge 0$: i) $w^t = v^t= \nabla h^*(-\rvA ^* y^t)$ and ii) from the definition of $f$ in~\eqref{eq:log-sum-exp}, 
\begin{equation}
s_j^t = \frac{\exp(a_j^\top x^t/\eta)}{\sum_{l\in[m]}\exp(a_l^\top x^t/\eta)} = \nabla_j f(Ax^t) = g_j^t, \quad \forall\, j\in[m], 
\end{equation}
and hence $s^t = g^t= \nabla f(Ax^t)$. 
As a result, we see that  D-LFP also has the same linear convergence rate in~\eqref{eq:lin_rate} as Algorithm~\ref{algo:GFWDA}, if we choose the stepsizes $\{\alpha_t\}_{t\ge 0}$  in the same way as in Theorem~\ref{thm:lin_conv} with $\kappa := \normt{A}_{1,\infty}^2/\eta^2$, which is the condition number of~\eqref{eq:P_LFP}. (Recall that $\normt{A}_{1,\infty}$ denotes the $(1,\infty)$-operator norm of $A$, and is given by  $\normt{A}_{1,\infty}:= {\sup}_{\normt{z}_{1}=1}\; \normt{Az}_\infty = {\max}_{j\in[m],i\in[n]}\, \abst{a_{j,i}},$ where $a_{j,i}$ denotes the $(j,i)$-th entry of $A$, for $j\in[m]$ and $i\in[n]$.) 



\subsection{Local Convergence Rate Analysis of LFP}

Now, let us analyze the local convergence rate of LFP (i.e., Algorithm~\ref{algo:LFP}), by regarding it as certain ``stochastic approximation'' of D-LFP. 
Specifically, we can rewrite the iterations~\eqref{eq:upd_x} and~\eqref{eq:upd_y} in Algorithm~\ref{algo:LFP} as 
\begin{alignat}{2}
x^{t+1}&:= (1-\alpha_t)x^t + \alpha_t (w^t  + \zetax^t) = x^t + \alpha_t (w^t  + \zetax^t-x^t),&&\quad\mbox{where}\quad \zetax^t:= e_{i_{t+1}} - w^t, \label{eq:zetax_t}\\
 y^{t+1}&:= (1-\alpha_t)y^t + \alpha_t (s^t +  \zetay^t) = y^t + \alpha_t (s^t +  \zetay^t - y^t),&&\quad\mbox{where}\quad \zetay^t:= e_{j_{t+1}} - s^t.\label{eq:zetay_t}
\end{alignat}
Note that $\zetax^t$ and $\zetay^t$ can be regarded as the stochastic errors resulted from the sampling steps  $i_{t+1}\sim w^t$ and $j_{t+1}\sim s^t$, respectively. In fact, we can easily see that the sequences of errors  $\{\zetax^t\}_{t\ge 0}$ and $\{\zetay^t\}_{t\ge 0}$ are martingale difference sequences. 
 Formally, let us define a filtration $\{\calF_t\}_{t\ge 0}$ such that for all $t\ge 0$, $\calF_t:= \sigma\big(\{(x_i,y_i)\}_{i=0}^t\big)$, namely the $\sigma$-field generated by the set of random variables $\{(x_i,y_i)\}_{i=0}^t$. Then we have 
\begin{align}
\bbE_t[{\zetax^t}]: =\bbE[{\zetax^t}\,|\,\calF_t] = 0 \quad \mbox{and}\quad \bbE_t[{\zetay^t}]: =\bbE[{\zetay^t}\,|\,\calF_t] = 0, \quad \forall\,t\ge 0.\label{eq:unbiased}
\end{align}

Next, let us  note that  since $(x^*,y^*)\in \ri\Delta_n\times\ri\Delta_m$ (cf.\ Theorem~\ref{thm:asymp_as}), there exist radii $r_\rvx,r_\rvy>0$ such that $\calB_{r_\rvx}(x^*)\times \calB_{r_\rvy}(y^*)\subseteq \ri\Delta_n\times\ri\Delta_m$, where 
\begin{align*}
\calB_{r_\rvx}(x^*)&:= \{x\in\bbR^n:e^\top x=1,\;\normt{x-x^*}_1\le \rx\}, \\
\calB_{r_\rvy}(y^*)&:= \{y\in\bbR^m:e^\top y=1,\;\normt{y-y^*}_1\le \ry\}. 
\end{align*}
For notational convenience, let us write $\calN(x^*,y^*):= \calB_{r_\rvx}(x^*)\times \calB_{r_\rvy}(y^*)$, which is a compact neighborhood of $(x^*,y^*)$ that is {\em bounded away} from the relative boundary of $\Delta_n\times\Delta_m$. This neighborhood will play an important role in our local convergence rate analysis of LFP. The advantage of this  neighborhood can be seen from the following lemma.

\begin{lemma}\label{lem:Lx_Ly}
There exist {finite} constants $\Lx,\Ly\ge 0$ such that 
\begin{alignat}{2}
\hx(x')&\le \hx(x) + \lranglet{\nabla \hx(x)}{x'-x} + (\Lx/2)\normt{x'-x}_1^2,\quad &&\forall\,x,x'\in \calB_{r_\rvx}(x^*), \label{eq:smooth_hx}\\
\hy(y')&\le \hy(y) + \lranglet{\nabla \hy(y)}{y'-y} + (\Ly/2)\normt{y'-y}_1^2,\quad &&\forall\,y,y'\in \calB_{r_\rvy}(y^*). \label{eq:smooth_hy}
\end{alignat} 
\end{lemma}
\begin{proof}
Indeed, we can set 
$\Lx:= {\max}_{x\in \calB_{r_\rvx}(x^*)}\,\normt{\nabla^2 \hx(x)}_{1,\infty}, $ which is finite since $\nabla^2 \hx$ is continuous on the compact set $\calB_{r_\rvx}(x^*)$. Similarly, we can set 
$\Ly:= {\max}_{y\in \calB_{r_\rvy}(y^*)}\,\normt{\nabla^2 \hy(y)}_{1,\infty}<+\infty $. 
\end{proof}
In addition, let us make another important observation: with high probability, the sequence $\{(x^t,y^t)\}_{t\ge 0}$ produced by LFP (i.e., Algorithm~\ref{algo:LFP}) will eventually lie inside $\calN(x^*,y^*)$, for any initial actions $i_0\in[n]$ and $j_0\in[m]$, and any step-sizes $\{\alpha_t\}_{t\ge 0}$ satisfying the conditions in~\eqref{eq:alpha_t_cond}. In fact, this is a simple corollary of Theorem~\ref{thm:asymp_as}, which is stated as follows. 

\begin{corollary}\label{cor:whp}
Define the sequence of events $\{A_{T}\}_{T\ge 0}$ such that
\begin{equation}
\calA_T:= \big\{\forall\,t\ge T,\;\; (x^t,y^t)\in \calN(x^*,y^*)\big\}, \quad\forall\,T\ge 0.  \label{eq:event_A}
\end{equation}
If the step-sizes $\{\alpha_t\}_{t\ge 0}$ satisfy~\eqref{eq:alpha_t_cond}, then for any $\delta\in(0,1)$, there exists $T(\delta)<+\infty$ such that $\Pr\big(\calA_{T(\delta)}\big)\ge 1-\delta.$
\end{corollary}
\begin{proof}
Indeed, from standard results (e.g.,~\cite[Theorem~3.3]{Hunter_06}), we know that the almost sure convergence result in~\eqref{eq:as_conv} 
 is equivalent to $\lim_{T\to\infty}\Pr(A_T) = 1$. Therefore, for any $\delta\in(0,1)$, there exists $T(\delta)<+\infty$ such that for all $T\ge T(\delta)$, $\Pr(A_T)\ge 1-\delta$. This completes the proof. 
\end{proof}

Lastly, our analysis requires the following technical lemma, whose proof follows from standard techniques 
(see e.g.,~\cite[Section~3]{Freund_16}). For completeness, we provide its proof in Appendix~\ref{app:proof}. 

\begin{lemma} \label{lem:recur_V}
Let $\{V_t\}_{t\ge 0}$ be a nonnegative sequence that satisfies the following recursion:
\begin{equation}
V_{t+1}\le (1-\alpha_t)V_t + \alpha_t^2 C,\quad\forall\,t\ge t_0 \qquad \mbox{for some } t_0\ge 0, \label{eq:recur_V}
\end{equation}
where  $C\ge 0$ and $\alpha_t\in[0,1]$ for all $t\ge 0$. If we choose $\alpha_t = {2}/({t+2})$ for all $t\ge 0$, then we have 
\begin{alignat}{3}
V_t &\le \frac{t_0(t_0+1)}{t(t+1)} V_{t_0} + \frac{4C(t-t_0)}{t(t+1)}, \qquad &&\forall\,t\ge t_0+1.
\end{alignat}
\end{lemma}

Equipped with the results above, we are ready to analyze the local convergence rate of LFP.
Indeed, our analysis of LFP 
modifies the analysis of D-LFP (i.e., Algorithm~\ref{algo:GFWDA}), by properly handling the stochastic errors $\zetax^t$ and $\zetay^t$ that appear in~\eqref{eq:zetax_t} and~\eqref{eq:zetay_t}, respectively. Before presenting our results, 
 let us first observe that the duality gap $\Delta(\cdot,\cdot)$ in~\eqref{eq:Delta_LFP} is jointly continuous on  $\Delta_n\times\Delta_m$, and hence we can define its maximum on $\Delta_n\times\Delta_m$ as  
\begin{equation}
\Delta_{\max}:={\max}_{(x,y)\in \Delta_n\times\Delta_m} \; \Delta(x,y) <+\infty.  \label{eq:Delta_max}
\end{equation} 


\begin{theorem}[Local convergence rate of LFP]\label{thm:stoc}
In Algorithm~\ref{algo:LFP}, choose any initial actions $i_0\in[n]$ and $j_0\in[m]$, and $\alpha_t = {2}/({t+2})$ for all $t\ge 0$. 
Then  for any $\delta\in(0,1)$,  there exists $T(\delta)<+\infty$ such that $\Pr(\calA_{T(\delta)})\ge 1-\delta$ 
and
\begin{align}
\bbE\big[\Delta\big(x^t,y^{t}\big)\big|\calA_{T(\delta)}\big] &\le \frac{T(\delta)(T(\delta)+1)}{t(t+1)} \Delta_{\max} + \frac{8\eta( \Lx+\Ly+2\kappa)(t-T(\delta))}{t(t+1)}, \;\; \forall\,t\ge T(\delta)+1, \label{eq:stoc_2/(t+2)}
\end{align}
where recall that $\kappa=\normt{A}_{1,\infty}^2/\eta^2$ and $\Lx$ and $\Ly$ satisfy~\eqref{eq:smooth_hx} and~\eqref{eq:smooth_hy}, respectively.  
\end{theorem}

\begin{proof}
Since the step-sizes $\{\alpha_t\}_{t\ge 0}$ satisfy the conditions in~\eqref{eq:alpha_t_cond}, from Corollary~\ref{cor:whp}, we see that  there exists $T(\delta)<+\infty$ such that $\Pr(\calA_{T(\delta)})\ge 1-\delta$. Now, by conditioning on the event $\calA_{T(\delta)}$, we see that $(x^t,y^t)\in \calN(x^*,y^*)$ for all $t\ge T(\delta)$. Thus, using~\eqref{eq:smooth_hx}, 
we have for all $t\ge T(\delta)$,  
\begin{align*}
\hx(x^{t+1}) - \hx(x^t)&\le \lranglet{\nabla \hx(x^t)}{x^{t+1}-x^t} + ({\Lx}/{2})\normt{x^{t+1} - x^t}_1^2\\
&= \alpha_t\lranglet{\nabla \hx(x^t)}{e_{i_{t+1}} - x^t} + \alpha_t^2(\Lx/{2}) \normt{e_{i_{t+1}} - x^t}_1^2\\
&\le \alpha_t\lranglet{\nabla \hx(x^t)}{w^t - x^t} + \alpha_t\lranglet{\nabla \hx(x^t)}{\zetax^t} + 2\Lx\alpha_t^2 \nt\label{eq:ub_hx0}\\
&\le  \alpha_t(\hx(w^t) - \hx(x^t)) + \alpha_t\lranglet{\nabla \hx(x^t)}{\zetax^t} + 2\Lx\alpha_t^2,\nt\label{eq:ub_hx}
\end{align*}
where~\eqref{eq:ub_hx0} follows from the definition of $\zetax^t$ in~\eqref{eq:zetax_t} 
and $\normt{e_{i_{t+1}} - x^t}_1^2\le 2(\normt{e_{i_{t+1}}}_1^2 + \normt{x^t}_1^2) = 4$,  
and~\eqref{eq:ub_hx} follows from the convexity of $\hx$. 
Similarly, we have 
\begin{align*}
\hy(y^{t+1}) - \hy(y^t)&\le \lranglet{\nabla \hy(y^t)}{y^{t+1} - y^t} + (\Ly/2) \normt{y^{t+1} - y^t}_1^2\\
&\le \alpha_t\lranglet{\nabla \hy(y^t)}{s^t +  \zetay^t - y^t} + 2{\Ly}\alpha_t^2\\
&\le \alpha_t(\hy(s^t) - \hy(y^t)) + \alpha_t\lranglet{\nabla \hy(y^t)}{\zetay^t} + 2{\Ly}\alpha_t^2.\nt \label{eq:ub_hy}
\end{align*}
In addition, from~\eqref{eq:P_LFP} and~\eqref{eq:D_LFP}, we see that $f$ and $h^*$ are differentiable with  $\eta^{-1}$-Lipschitz gradients on $\bbR^n$ and $\bbR^m$, respectively, and hence 
\begin{align}
f(Ax^{t+1})-f(Ax^{t})&\le  \alpha_t\lranglet{\nabla f(Ax^{t})}{A(e_{i_{t+1}}-x^t)} + \alpha_t^2(\eta^{-1}/2) \normt{A}_{1,\infty}^2\normt{e_{i_{t+1}}-x^t}_1^2 \nn\\
&\le  \alpha_t(\lranglet{s^t}{A(w^t-x^t)}+\lranglet{s^t}{A\zetax^t})+ \alpha_t^2(2\normt{A}_{1,\infty}^2/\eta), \label{eq:smooth_gx_A_stoc}\\
h^*(-A^\top y^{t+1})-h^*(-A^\top y^{t}) &\le - \alpha_t\lranglet{\nabla h^*(-A^\top y^{t})}{A^\top (e_{j_{t+1}}-y^{t})} + \alpha_t^2(\eta^{-1}/2) \normt{A}_{1,\infty}^2\normt{e_{j_{t+1}}-y^t}_1^2,\nn\\
&\le  - \alpha_t(\lranglet{w^t}{A^\top (s^t-y^{t})} + \lranglet{w^t}{A^\top \zetay^t} )  + \alpha_t^2(2\normt{A}_{1,\infty}^2/\eta), \label{eq:smooth_gy_A_stoc}
\end{align}
where we use $s^t = \nabla f(Ax^{t})$ and $w^t = \nabla h^*(-A^\top y^t)$ (cf.\ Section~\ref{sec:D_LFP}). 
Therefore, by combining~\eqref{eq:ub_hx} to~\eqref{eq:smooth_gy_A_stoc}, and use the definitions $h := \eta\hx + \iota_{\Delta_n}$ and $f^*:= \eta\hy + \iota_{\Delta_m}$ in~\eqref{eq:P_LFP} and~\eqref{eq:D_LFP}, respectively,  we have 
\begin{equation}
\begin{split}
\Delta(x^{t+1},y^{t+1}) - \Delta(x^{t},y^{t}) &\le  \alpha_t\big\{h(w^t)+ \lranglet{A w^t}{ y^t} -h(x^t)  + f^*(s^t)-\lranglet{s^t}{A x^t} -f^*(y^t)\big\}\\
&+ \alpha_t\big\{\lranglet{\nabla h(x^t)+ A^\top s^t}{\zetax^t}  +\lranglet{\nabla f^*(y^t)-A w^t}{\zetay^t} \big\}\\
&+ 2\alpha_t^2\eta\left\{ \Lx+\Ly+2\normt{A}_{1,\infty}^2/\eta^2\right\}, \
\end{split}  \label{eq:Delta_recur_LFP}
\end{equation}
Since $s^t = \nabla f(Ax^{t})$ and $w^t = \nabla h^*(-A^\top y^t)$, we know that 
\begin{equation}
f^*(s^t)-\lranglet{s^t}{A x^t} = -f(Ax^t) \quad \mbox{and}\quad h(w^t)+ \lranglet{A w^t}{ y^t}= -h^*(-A^\top y^t).  \label{eq:Fenchel_equality}
\end{equation}
By combining~\eqref{eq:Delta_recur_LFP},~\eqref{eq:Fenchel_equality} and~\eqref{eq:unbiased}, we know that
\begin{align*}
\bbE_t[\Delta(x^{t+1},y^{t+1})]&\le (1-\alpha_t)\Delta(x^{t},y^{t}) + 2\alpha_t^2\eta( \Lx+\Ly+2\kappa), \quad  \forall\,t\ge T(\delta). \nt \label{eq:V_t+1_V_t_stoc}
\end{align*}
Finally, by applying Lemma~\ref{lem:recur_V} to~\eqref{eq:V_t+1_V_t_stoc} and using the definition of $\Delta_{\max}$ in~\eqref{eq:Delta_max}, we 
complete the proof. 
\end{proof}

\section{Preliminary Experimental Studies}\label{sec:experiments}

{\bf Experimental setup.} We compare the numerical performance of several previously mentioned methods on the~\eqref{eq:P_LFP} problem. These methods include 
\begin{enumerate}[label=\roman*)]
\item {\sf GFW-N}: The GFW method (i.e., Algorithm~\ref{algo:GFW}) with decreasing step-sizes  in Nesterov~\cite[Section~5]{Nest_18}. Specifically, $\alpha_t = \frac{6(t+1)}{(t+2)(2t+3)}$ for $t\ge 0$.
\item {\sf GFW-G}: The GFW method (i.e., Algorithm~\ref{algo:GFW}) with constant step-sizes  in Ghadimi~\cite[Corollary~1(b)]{Ghad_19}. Specifically,  $\alpha_t = 1/(1+4\kappa)$ for $t\ge 0$, where $\kappa=\normt{A}_{1,\infty}^2/\eta^2$. 
\item {\sf GFWDA}: Algorithm~\ref{algo:GFWDA} (or equivalently, Algorithm~\ref{algo:det_LFP}) with constant step-sizes as in Theorem~\ref{thm:lin_conv}.  Specifically,  $\alpha_t = \min\{1/(2\kappa),1\}$ for $t\ge 0$, where $\kappa=\normt{A}_{1,\infty}^2/\eta^2$.  
\item {\sf LFP}: Algorithm~\ref{algo:LFP} with decreasing step-sizes as in Theorem~\ref{thm:stoc}. Specifically,  $\alpha_t = 2/(t+2)$ for $t\ge 0$.
\end{enumerate}
To generate the data matrix $A$, we choose the dimensions $m = 100$ and $n=200$, and generate each entry of $A$ independently from the uniform distribution on the interval $[-8, 8]$. In addition, we choose $\eta = 10$. For the specific instance of $A$ used in our experiments, we have $\normt{A}_{1,\infty}\approx 8.0$ and hence $\kappa=\normt{A}_{1,\infty}^2/\eta^2 \approx 0.64.$ \\

\noindent {\bf Comparison criterion and starting points.} Note that each of the four methods above is able to generate certain sequence of duality gaps that converges to zero.  Specifically, the sequence generated by {\sf GFW-N} and {\sf GFW-G} is $\{\barG_t\}_{t\ge 0}$ (cf.~\eqref{eq:FW_h_nsc}) and  the sequence generated by {\sf GFWDA} and {\sf LFP} is $\{\Delta(x^t,y^t)\}_{t\ge 0}$ (cf.~\eqref{eq:pd_gap}). Due to this, we will use the convergence speed of these duality gaps as the comparison criterion. For starting points, we choose $x^0 = e_1$ for all the four methods. In addition, for {\sf GFWDA}, we choose $y^0 = \nabla f(\rvA x^0)$, so that {\sf GFW-N}, {\sf GFW-G} and {\sf GFWDA} have the same initial duality gap, namely $\barG_0 = G(x^0) = \Delta(x^0,\nabla f(\rvA x^0)) = \Delta(x^0,y^0)$. As for {\sf LFP}, since we need to choose $y^0 = e_{j_0}$ for some $j_0\in[m]$ (cf.\ Algorithm~\ref{algo:LFP}), we let $j_0 = \argmax_{j\in[m]} \nabla_j f(\rvA x^0)$. Note that this choice of $y^0$ will result in a larger duality gap compared to the one given by $y^0 = \nabla f(\rvA x^0)$. However, in our experiments, we observe that the difference is not significant. \\

\noindent {\bf Experimental results.}  We plot the duality gaps generated by all the four methods versus iterations  in Figure~\ref{fig:comparison}, in both log-linear and log-log scales. Since {\sf LFP} is a stochastic algorithm, we repeatedly run it for 10 times (with the same starting points as described above) and plot the averaged duality-gap trajectories. From Figure~\ref{fig:comparison}, we can make the following observations. First, {\sf GFWDA} and {\sf LFP} are the fastest and slowest among all the four methods, respectively. In fact, {\sf GFWDA} produces a duality gap of order $10^{-14}$ in less than 15 iterations,  while  {\sf LFP} hardly makes any progress during the first 30 iterations. Second, {\sf GFW-N} converges at a sub-linear rate that is much faster than 
$O(1/t^2)$, which is derived from theory (cf.~\eqref{eq:FW_Nest}). This is probably because~\eqref{eq:P_LFP} possesses certain structural properties (other than smoothness and strong convexity) that are favorable to {\sf GFW-N}. Third, although both {\sf GFWDA} and {\sf GFW-G}  converge linearly, the linear rate of {\sf GFW-G} is slower than that of  {\sf GFWDA}. This indeed agrees with our theoretical analysis. Specifically, 
the linear rate of {\sf GFW-G} is $1-1/(2(1+4\kappa))$ (cf.~\eqref{eq:FW_Ghad}), which is slower than the linear rate of {\sf GFWDA}, namely $1-{1}/(4\kappa)$  (cf.\ Theorem~\ref{thm:lin_conv}). 

 \begin{figure}[t!]\centering 
\subfloat[Log-linear scale]{\includegraphics[width=.45\linewidth]{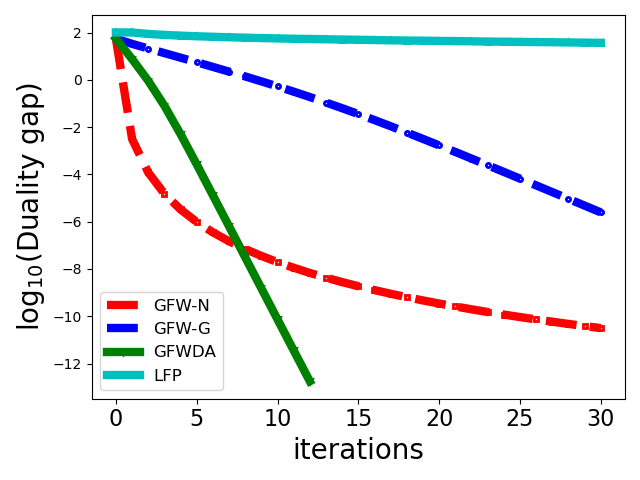}}\hspace{2ex}
\subfloat[Log-log scale]{\includegraphics[width=.45\linewidth]{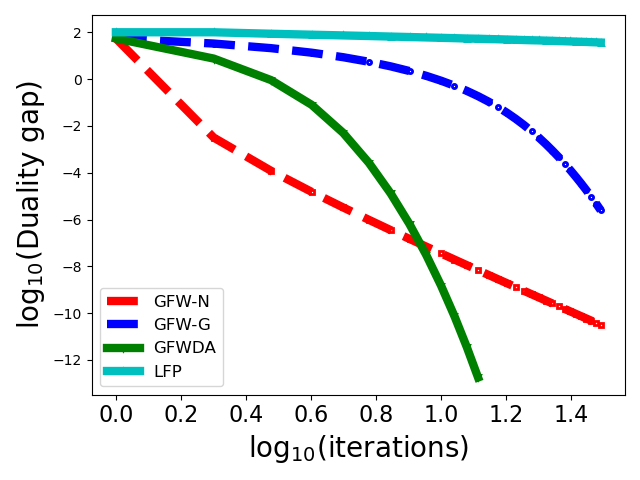}} 
\caption{Comparison of the convergence speed of duality gaps generated by {\sf GFW-N}, {\sf GFW-G}, {\sf GFWDA} and {\sf LFP} in (a) log-linear scale and (b) log-log scale.}\label{fig:comparison}
\end{figure}

\begin{figure}[t!]\centering 
\begin{minipage}[t!]{0.45\linewidth}
\includegraphics[width=\linewidth]{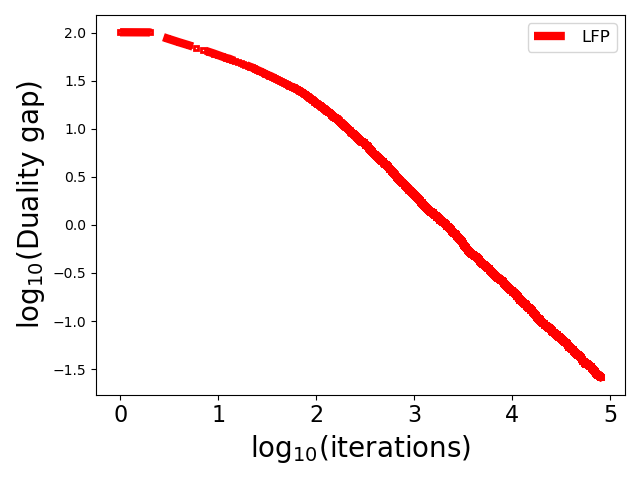}\centering
\vspace{-1em}
\captionof{figure}{Log-log plot of the averaged duality-gap trajectories generated by {\sf LFP}.}\label{fig:LFP}
\end{minipage}\hspace{4ex}
\begin{minipage}[t!]{0.5\linewidth}
\vspace{.5em}
\setlength\tabcolsep{2ex}
\def\arraystretch{1.3}
\begin{tabular}{|c|c|} 
\hline
Iteration intervals & Slopes 
\\\hline
$1$ to $5$ & -0.229\\
$5$ to $5^2$& -0.342\\
$5^2$ to $5^3$ & -0.590\\
$5^3$ to $5^4$ & -0.951\\
$5^4$ to $5^5$ & -1.030\\
$5^5$ to $5^6$ & -0.994\\
$5^6$ to $5^7$ & -0.991\\
\hline
\end{tabular}\centering
\captionof{table}{Slopes of the plot in Figure~\ref{fig:LFP} over iteration intervals of constant lengths in log-scale.} \label{tab:slopes}
\end{minipage}
\end{figure}

Next, let us examine the local convergence rate of {\sf LFP}. From the plot in Figure~\ref{fig:LFP}, we can observe an $O(1/t)$ convergence rate starting from around $100$ iterations. For better illustration, in Table~\ref{tab:slopes}, we compute the slopes of this plot over iteration intervals of constant lengths in log-scale. From Table~\ref{tab:slopes}, we can clearly see that the magnitudes of the slopes are initially very small, which correspond to the slow initial convergence of the duality gaps. However, they gradually converge to one, which correspond to the $O(1/t)$ local convergence rate as ``predicted'' in Theorem~\ref{thm:stoc}.  


\section*{Acknowledgment} 
The first author's research is supported by  AFOSR Grant No. FA9550-19-1-0240.

\appendix

\section{Proof of Lemma~\ref{lem:recur_V}}\label{app:proof}
Let $\{\beta_t\}_{t\ge 0}$ be a nonnegative auxiliary sequence such that
\begin{equation}
\beta_t\ge \beta_{t+1}(1-\alpha_{t+1}),\quad \forall\,t\ge t_0. \label{eq:beta_t}
\end{equation}
As a result, we have
\begin{align}
\beta_{t+1}(1-\alpha_{t+1})V_{t+1}\le \beta_{t}(1-\alpha_{t})V_{t} + \beta_t\alpha_t^2C, \quad \forall\,t\ge t_0. \label{eq:recur_beta}
\end{align}
After telescoping~\eqref{eq:recur_beta} over $i = t_0, \ldots, t-1$, 
we have 
\begin{align*}
&\beta_{t}(1-\alpha_{t})V_{t} \le \beta_{t_0}(1-\alpha_{t_0})V_{t_0} + C\textstyle\sum_{i=t_0}^{t-1} \beta_i\alpha_i^2,\qquad \forall\,t\ge t_0+1,\\
\end{align*}
and consequently,
\begin{align}
 V_t \le \frac{\beta_{t_0}(1-\alpha_{t_0})V_{t_0}}{\beta_{t}(1-\alpha_{t})} + \dfrac{C\sum_{i=t_0}^{t-1} \beta_i\alpha_i^2}{\beta_{t}(1-\alpha_{t})},\qquad \forall\,t\ge t_0+1, \label{eq:V_t0}
\end{align}
If we choose $\alpha_t = 2/(t+2)$ for all $t\ge 0$, to satisfy~\eqref{eq:beta_t}, we can then choose $\beta_t = (t+2)(t+1)/2$ for all $t\ge 0$. Therefore, we have 
\begin{equation}
\beta_{t}(1-\alpha_{t}) = t(t+1)/2 \quad\mbox{and}\quad \textstyle\sum_{i=t_0}^{t-1} \beta_i\alpha_i^2 = \sum_{i=t_0}^{t-1}  {2(i+1)}/{(i+2)}\le 2(t-t_0). \label{eq:V_t0_2}
\end{equation}
Substituting~\eqref{eq:V_t0_2} into~\eqref{eq:V_t0}, we then complete the proof.

\bibliographystyle{IEEEtr}
\bibliography{stat_ref,stoc_ref,math_opt,mach_learn}

\end{document}